\documentclass[11pt]{amsart}
\usepackage[T1]{fontenc}
\usepackage{amsmath}
\usepackage{mathrsfs}
\usepackage{amssymb}
\usepackage{framed}
\usepackage{graphicx}

\newtheorem{theorem}{Theorem}[section]
\newtheorem{lemma}[theorem]{Lemma}
\theoremstyle{definition}
\newtheorem{definition}[theorem]{Definition}
\newtheorem{example}[theorem]{Example}

\theoremstyle{remark}
\newtheorem{remark}[theorem]{Remark}
\theoremstyle{corollary}
\newtheorem{corollary}[theorem]{Corollary}
\theoremstyle{proposition}
\newtheorem{proposition}[theorem]{Proposition}

\usepackage{epsfig}
\usepackage{hyperref}

\newcommand{\R}{\mathbb{R}}

\newcommand{\Cn}{\mathbb{C}^n}
\newcommand{\C}{\mathbb{C}}

\newcommand{\Z}{\mathbb{Z}}

\newcommand{\N}{\mathbb{N}}

\newcommand{\abs}[1]{\left|#1\right|}

\newcommand{\Dbar}{\bar D}

\DeclareMathOperator{\real}{Re}
\DeclareMathOperator{\imag}{Im}
\newcommand{\im}{\operatorname{Im}}
\newcommand{\re}{\operatorname{Re}}

\usepackage{mathtools}
\mathtoolsset{showonlyrefs}

\begin{document}
\title{Level sets of certain subclasses of $\alpha$-analytic functions}

\author[Abtin~Daghighi]{Abtin~Daghighi}
\address{Department of Mathematics, Mittuniversitet, Holmgatan 10, SE-851 70 Sundsvall, Sweden}
\email{Abtin.Daghighi@miun.se}
\author[Frank~Wikstr{\"o}m]{Frank~Wikstr{\"o}m }
\address{Center for Mathematical Sciences, Lund University, Box 118,
SE-221 00 Lund, Sweden}
\email{Frank.Wikstrom@math.lth.se}

\begin{abstract}
For an open set $V\subset\Cn$, denote by $\mathscr{M}_{\alpha}(V)$
  the family of $\alpha$-analytic functions that obey a boundary
  maximum modulus principle. We prove that, on a bounded domain
  $\Omega\subset \Cn$, with continuous boundary (that in each variable separately allows a solution to the Dirichlet problem), a function $f\in
  \mathscr{M}_{\alpha}(\Omega\setminus f^{-1}(0))$ automatically
  satisfies $f\in \mathscr{M}_{\alpha}(\Omega)$, if it is
  $C^{\alpha_j-1}$-smooth in the $z_j$ variable, $\alpha\in \Z_+^n$ up to the boundary.
  For a submanifold $U\subset \Cn$, denote by
  $\mathfrak{M}_{\alpha}(U),$ the set of functions locally
  approximable by $\alpha$-analytic functions where each approximating
  member and its reciprocal (off the singularities) obey the boundary
  maximum modulus principle.  We prove, that for a $C^3$-smooth
  hypersurface, $\Omega$, a member of $\mathfrak{M}_{\alpha}(\Omega)$,
  cannot have constant modulus near a point where the Levi form has a
  positive eigenvalue, unless it is there the trace of a polyanalytic
  function of a simple form. 
\end{abstract}

\subjclass[2000]{Primary 
35G05,  
32A99
} 
\keywords{Polyanalytic functions, $q$-analytic functions, zero sets, level sets, $\alpha$-analytic functions}
\maketitle

\section{Introduction}

A higher order generalization of the holomorphic functions are the
solutions of the equation $\frac{\partial^q}{\partial \bar{z}^q}f=0$,
for a positive integer $q$. These functions are called
\emph{polyanalytic functions of order $q$} or simply
\emph{$q$-analytic functions}.  An excellent introduction to
polyanalytic functions can be found in the survey article by
Balk~\cite{ca1}. 
%
%
A higher dimensional generalization of $q$-analytic
functions is the set of \emph{$\alpha$-analytic functions}, $\alpha\in
\Z_+^n.$ In this paper, we consider the set of function which can be
locally uniformly approximated by $\alpha$-analytic functions
satisfying a boundary maximum modulus principle. We prove an extension
of Rad\'o's theorem for such functions on complex manifolds.
Secondly, we prove, for a special subclass on (not necessarily
complex) submanifolds, a generalization of the fact that $q$-analytic
functions cannot have constant modulus on open sets unless they are of
a particularly simple form.

\subsubsection*{Results} 

Our main results are Theorem~\ref{radoprop1},
Theorem~\ref{propettconv} and Theorem~\ref{jo}.
Theorem~\ref{propettconv} is a generalization of the fact that
$q$-analytic functions cannot have constant modulus on open sets
unless they are of the form $\lambda \overline{Q(z)}/Q(z)$ for some
polynomial $Q$ of degree $<q$ and some constant $\lambda\in\C$. In
particular we consider families, 
that can be locally uniformly approximated by $\alpha$-analytic functions, which themselves and and their reciprocals (except at singularities) satisfy a boundary maximum modulus principle, and we
show that such families cannot have members with constant modulus near
points with at least one positive Levi eigenvalue, unless these
members are the there the trace of an $\alpha$-analytic function of
the form $\lambda \overline{Q(z)}/Q(z)$ for some holomorphic
polynomial $Q(z)$ of degree $<\alpha_j$ in $z_j,1\leq j\leq n$ and
constant $\lambda$. For hypersurfaces, this is done for the
$C^3$-smooth case but we also give a partial generalization for
$C^4$-smooth real submanifolds of higher codimension, see
Theorem~\ref{jo}.
Theorem~\ref{radoprop1}, is that functions in $n$ complex variables
which are $C^{\alpha_j-1}$-smooth in the $z_j$ variable, and which,
off their zero set, are $\alpha$-analytic functions that obey the
boundary maximum modulus principle, extend across their zero set to
functions of the same class.

\section{Preliminaries}

For general background on polyanalytic functions, see Balk~\cite{ca1}
and references therein (in particular we want to mention the earlier
work by Balk \& Zuev~\cite{balk}).

\begin{definition}
  Let $U\subset \C$ be a domain. A function $f : U \to \C$
  is called \emph{$q$-analytic} (or \emph{polyanalytic of order $q$}) if
  it can be written as
  \begin{equation}
    f(z)=\sum_{j=0}^{q-1} a_j(z)\bar{z}^j,a_j\in \mathscr{O}(U).
  \end{equation}
  If $a_{q-1}\not\equiv 0$, then $q$ is called the \emph{exact}
  order.\index{Exact order}
  
  It is well-known, and almost self-evident, that $f(z)=u(x,y)+iv(x,y)$
  of class $C^q(U)$ is $q$-analytic on $U$ if and only if
  \begin{equation}
    \frac{\partial^q f}{\partial\bar{z}^q} = 0 \quad\text{on $U$}.
  \end{equation}
  (see Balk~\cite[p.\,198]{ca1}).
\end{definition}

The $q$-analytic functions behave well under locally uniform convergence.

\begin{proposition}[Balk~{\cite[p.\,206]{ca1}}]\label{unifconv}
  {\it Let $U\subset \C$ be a domain and let $\{f_j\}_{j\in \N}$ be
  polyanalytic functions of the same order $q$, on $U$, such that the
  $f_j$ converge uniformly on $U.$ Then the limit is a polyanalytic
  function of order $q$ on $U$.}
\end{proposition}

We cannot directly generalize the fact that a holomorphic function
with constant modulus on an open subset must reduce to a constant, to
the case of $q$-analytic functions. A characterization is, however,
known due to Balk~\cite{balle}.

\begin{theorem}[Balk~\cite{balle}]\label{ballcar}
  A polyanalytic function of order $q$ in a domain $\Omega\subset\C$
  has constant modulus if and only if $f$ is representable in the form
  $f(z)=\lambda \cdot \overline{P(z)}/P(z)$, where $P(z)$ is a
  polynomial of degree at most $q-1$, and $\lambda \in \C$ is a
  constant.
\end{theorem}

Note that, in particular, the only \emph{entire} polyanalytic
functions with constant modulus are the constant functions.

\begin{definition}[See Balk \cite{balk1965}]\label{deff0}
  Let $U\subset\C$ be a domain and let $p\in E\subset U$.  By
  definition the line $\ell :=\{z\in \C\colon
  z=p+te^{i\theta},\abs{t}<\infty, t\in \R\}$, $p$ and $\theta$
  constants, is a \emph{limiting direction of the set $E$ at $p$} if
  $E$ contains a sequence of points $z_j =p+t_j e^{i\theta_j},$
  $t_j\to 0,\theta_j\to \theta,t_j\neq 0.$ The point $p$ is called a
  \emph{condensation point of order $k$ of $E$} if there are $k$
  different lines through $p$ which are limiting directions of $E$.
\end{definition}

The following uniqueness property is known.

\begin{theorem}[Balk~{\cite[p.\,202]{ca1}}]\label{condens}
  Let $U\subset \C$ be a domain and let $f$ and $g$ be polyanalytic
  functions of order $q$ on $U$. Assume that $E\subset U$ and that $E$
  has a condensation point of order $q$. Then $f\equiv g$ on $E$
  implies $f\equiv g$ on $U$.
\end{theorem}

\subsection*{Polyanalytic functions of several variables}

Avanissian \& Traoré~\cite{avan1},~\cite{avan2} introduced the
following definition of a polyanalytic function of order $\alpha$ in
several variables.

\begin{definition}[Avanissian \& Traoré~\cite{avan1}]\label{avandef}
  Let $\Omega\subset \Cn$ be a domain and let $z=x+iy$ denote
  holomorphic coordinates in $\Cn$.  A function $f\in
  C^{\infty}(\Omega,\C)$ is called \emph{polyanalytic of order $\alpha$} if there exists
  a multi-index $\alpha\in \Z_+^n$ such that in a neighborhood of
  every point of $\Omega$, $\left(\frac{\partial}{\partial
      \bar{z}_j}\right)^{\alpha_j}\!f(z)=0,1\leq j\leq n$. If the
  integer $\alpha_j,1\leq j\leq n$ is minimal, then $f$ is said to be
  \emph{polyanalytic of exact order $\alpha$}.
\end{definition}

\begin{definition}
  A function $f$ is called \emph{countably analytic} on an open set
  $\Omega\subset \Cn$ if $f$ has a local expansion near every point
  $p\in \Omega$ of the form $f(z)=\sum_{\alpha}
  h_{\alpha}(z,p)(\bar{z}-\bar{p})^{\alpha}$ for holomorphic
  $h_{\alpha}(z,p)$ (where $p$ is fixed i.e.\ $h_{\alpha}$ is holomorhic in the variable $z$).
\end{definition}

\begin{definition}
  Let $\Omega\subset\Cn$ be an open subset and let $(z_1,\ldots,z_n)$
  denote holomorphic coordinates for $\Cn.$ A function $f$ is said to
  be \emph{separately $C^{k}$-smooth with respect to the
    $z_j$-variable}, if for any fixed $(c_1,\ldots,c_{n-1})\in
  \C^{n-1}$ the function
  $f(c_1,\ldots,c_{j-1},z_j,c_j,\ldots,c_{n-1})$ is $C^{k}$-smooth
  with respect to $\re z_j, \im z_j.$
\end{definition}

\section{$\alpha$-analyticity across zero sets for a special family} 
\label{radosec}

We cannot hope for a strong maximum principle for $q$-analytic
functions. Take for example the $2$-analytic function $f(z)=1-z\bar{z}$ on $\C,$
which not only attains a strict local maximum at the origin but in fact
vanishes on the boundary of the unit disc, thus it is certainly not
determined by its boundary values.

\begin{definition}\label{bmp}
  Let $\Omega\subset \Cn$ be a submanifold and let
  $\mathfrak{M}\subset C(\Omega,\C)$ be a family of functions.  Let
  $D$ denote the unit disc, $D:=\{\zeta\in \C:\abs{\zeta}<1\}$.  The
  family $\mathfrak{M}$ is said to have the \emph{one dimensional
    boundary maximum modulus property} (\textsc{1d-bmmp}) if given
  $f\in \mathfrak{M}$, then for any $\psi\in \mathscr{O}(D,\Omega)\cap
  C(\overline{D},\Omega)$,
  \begin{equation}
    \max_{\zeta\in \overline{D}}\abs{f\circ \psi(\zeta)}\leq 
    \max_{\zeta\in \partial D}\abs{f\circ \psi(\zeta)}.
  \end{equation}
\end{definition}
\begin{example}
  Let $\Omega\subset \Cn$ be a complex submanifold. Then
  $\mathscr{O}(\Omega)\subset C(\Omega,\C)$ clearly has the one
  dimensional boundary maximum modulus property.
\end{example}

Using the definition of Avanissian \& Traoré~\cite{avan1}, together
with Theorem~\ref{hartog1} we obtain the following Corollary to
Proposition~\ref{unifconv}.

We shall, in this section, use the following notation:
Let $U\subset\Cn$ be a submanifold. Denote by
$\mathscr{M}_{\alpha}(U)$ the set of restrictions to $U$ of
$\alpha$-analytic functions on $\Cn$ that obey the
one dimensional boundary maximum modulus property of
Definition~\ref{bmp}.

An immediate consequence of the fact that the set of holomorphic
functions obey a strong maximum principle, is that, if $U\subset\Cn$
is a complex submanifold (not necessarily of dimension $n$), then
$\mathscr{O}(U)\subseteq\mathscr{M}_{\alpha}(U)$, but more can be said
(see Example \ref{thx}). For simplicity of notation, let $\Dbar =
\frac{\partial}{\partial \bar{z}}$.
\\
\\
A complex function on $\Cn$ annihilated by the system
$\frac{\partial}{\partial\bar{z}_j}$, $1\leq j\leq n$, on
$\Cn\setminus f^{-1}(0)$ is automatically holomorphic.  This theorem
was proved for $n=1$ by Radó~\cite{t1} already in 1924, and
generalized to $n > 1$ by Cartan~\cite{cartan} in 1952. We shall
generalize this result to the subfamilies, 
$\alpha$-analytic functions which obey the boundary maximum modulus
principle.
Kaufman \cite{kauf} makes use of a maximum principle 
and combines this with an approximation property on the boundary to prove 
Radó's theorem and we shall use a similar method of proof for
the following generalization for the case of bounded 
domains. 

\begin{theorem}\label{radoprop}
  Let $\Omega\subset\C$ be a bounded domain with continuous boundary that allows a solution to the Dirichlet problem. Let $q$ be a positive
  integer and let $f \in C^{q-1}(\overline{\Omega})$ such that $f\in
  \mathscr{M}_q(\Omega\setminus f^{-1}(0)).$ Then $f\in
  \mathscr{M}_q(\Omega).$
\end{theorem}

\begin{proof}
  Note that $q=1$ corresponds to the well-known Rad\'o's theorem in one
  complex variable.

\begin{remark}\label{remarkrado}
  If $U\subset \Cn$ a submanifold and $G$ is a continuous function on
  $\overline{U}$ which satisfies the boundary maximum modulus principle on the
  (necessarily open) subset $U\setminus G^{-1}(0)$ then $G$ satisfies
  the boundary maximum modulus principle on $U.$ Indeed, let $V\Subset
  U,$ be a domain. If $V\cap G^{-1}(0)=V$ then, by continuity,
  $\max_{z\in \overline{V}}\abs{G(z)}=0,$ so we are done. If instead
  the (necessarily open) set $V\cap \{G\neq 0\}$ is nonempty, then
  $\max_{z\in \overline{V}}\abs{G(z)}=$ $\max_{z\in \overline{V\cap
      \{G\neq 0\}}}\abs{G(z)}=$ $\max_{z\in \partial V\cap \{G\neq
    0\}}\abs{G(z)}$ $=\max_{z\in \partial V}\abs{G(z)}.$
\end{remark}

By Remark \ref{remarkrado} it is sufficient to show that $f$ extends
to a $q$-analytic function on $\Omega.$ For the sake of clarity we
first prove the case $q=2,$ i.e.\ assume $f\in
\mathscr{M}_q(\Omega\setminus f^{-1}(0))$ (only small modifications
are then required to prove the cases $2<q<\infty$).  By Corollary
\ref{unifconv1} we have that $f$ is $2$-analytic on $\Omega\setminus
f^{-1}(0)$, and by definition we know that $f$ 
satisfies the boundary maximum modulus principle on $\Omega\setminus
f^{-1}(0).$ Set
\begin{equation}
  u(z):=\frac{\partial}{\partial\overline{z}} f(z)=\Dbar f(z).
\end{equation}
Since $f\in C^{2-1}(\Omega),$ the function $u$ is well-defined on
$\Omega.$ Let $U = \Omega \setminus f^{-1}(0)$ for convenience of
notation. By definition, $u$ is holomorphic on $U$. Furthermore, $u =
0$ on the interior of $f^{-1}(0)$ so,
\begin{equation}
  \Dbar u(z) = 0 \qquad \forall z \in \Omega \setminus \partial U.
\end{equation}

We shall need the following lemma.

\begin{lemma}\label{flem}
  Let $g$ be a function continuous on $\overline{U}$ and holomorphic
  on $U$.  Then for all $z\in U$,
\begin{equation}
  |g(z)| \leq \sup_{\zeta \in \partial U \cap \partial \Omega} |g(\zeta)|.
\end{equation}
\end{lemma}

\begin{proof}
  First of all, $\sup_{\partial U\cap \Omega^{\circ} }\left| f\right|
  =0$ since $f=0$ on the given set. Secondly, note that $f$ and $g^j$
  each satisfies the boundary maximum modulus principle applied to
  $U$.  Thus for any $z\in U,$
  \begin{equation}
    \left|f(z)\right| \left|g(z)\right|^j  \leq \left(\sup_{\partial U\cap \overline{\Omega} } \left| f\right|\right)\cdot \left(
      \sup_{\partial U\cap \overline{\Omega}} \left| g\right|
    \right)^j\leq 
    \left(\sup_{\partial U\cap \partial \Omega } \left| f\right|\right)\cdot \left(
      \sup_{\partial U\cap \partial \Omega} \left| g\right|
    \right)^j,
  \end{equation}
  which in turn implies, after taking $1/j$th power and the limit
  as $j\to \infty,$ $\left| g(z)\right|\leq \sup_{w\in \partial
    U\cap \partial \Omega } \left|g(w) \right|,$ $z\in
  \Omega\setminus \partial U.$
\end{proof}
By Lemma \ref{flem},
\begin{equation}\label{ek00}
\left| u(z)\right|\leq \sup_{w\in \partial U\cap \partial \Omega } \left|u (w) \right|, z\in U.
\end{equation}
Also, we know that $\abs{u(z)}=0$ for all $z$ in the interior of
$f^{-1}(0)$. In fact, given Lemma \ref{flem}, a verbatim repetition of
an argument which can be found in e.g.\ Kaufmann \cite{kauf}, proves
that $U$ must be a dense subset of $\Omega$, in particular $f^{-1}(0)$
must have \emph{empty} interior.  This together with
Equation~\eqref{ek00} gives,
\begin{equation}\label{ek01}
  \left| u(z)\right|\leq \sup_{w\in \partial U\cap \partial \Omega } \left|u(w) \right|, z\in \Omega\setminus \partial U.
\end{equation}
Next we show that $\partial u/\partial\bar{z}$ is harmonic, and since
it is zero on a dense open subset of $\Omega$ it vanishes identically.
To see that $\partial u/\partial\bar{z}$ is harmonic, set $u = w+iv,$
and show that $w,v$ are harmonic as follows: 
First of all we can find a sequence of complex polynomials whose real parts converge uniformly on
$\partial \Omega$ to $u$ (a procedure for doing this is described in e.g.\ Boivin \& Gauthier \cite{boivinn}, p.123. In short one solves the Dirichlet problem for the boundary, in order to obtain a harmonic function, say $W$, in $\Omega$ whose continuous boundary values agree with $u.$ Of course it is important to refer to the condition that the boundary be given by a continuous function. Then $W$ can be complemented with its harmonic conjugate in $\Omega$, to a holomorphic function on $\Omega.$ The latter can the be approximated by complex polynomials). 
So let $\{P_j\}_{j\in \N},$ be
a sequence of holomorphic polynomials such that $\real (P_j-u)\to 0$
on $\partial \Omega.$ We have $\abs{e^{P_j-u}}=e^{\real(P_j-u)},$ and
also that $e^{P_j-u}, e^{u-P_j}$ are holomorphic on $U$. Thus the
maximum principle of Equation~\eqref{ek01}) applies to both
$e^{P_j-u}$ and $e^{u-P_j}$. We can choose $P_j$ such that $\abs{P_j
  -u} < \frac{1}{j}$. Consequently,
\begin{equation}\label{ineq0}
  \abs{e^{(P_j(z) -u(z))}} < e^{\frac{1}{j} } \text{ and }\abs{e^{(u(z)- P_j(z))}} < e^{\frac{1}{j} },z\in \partial \Omega , 
\end{equation}
and by the maximum principle of Equation~\eqref{ek01}, the
inequalities \eqref{ineq0} continue to hold in $\Omega$. This however
implies that
\begin{equation}\label{ineqw00}
e^{\frac{1}{j}}>\abs{e^{P_j(z)-u(z)}}=e^{\real(P_j(z)-u(z))}, \forall z\in \Omega.
\end{equation}
After taking logarithms,
\begin{equation}\label{ineqw0}
\frac{1}{j}>\real(P_j(z)-u(z)),\forall z\in \Omega.
\end{equation}

Since the real part of each $P_j$ is harmonic, this uniform
convergence implies that $\real u(z)=w(z)$ is harmonic on
$\Omega$. Analogously one shows that $v$ is harmonic.  It then follows
that $\partial u/\partial\bar{z}$ has harmonic real and imaginary
parts since, $\Delta \frac{\partial u}{\partial\bar{z}} =
\frac{\partial }{\partial x}(\Delta w)- \frac{\partial }{\partial
  y}(\Delta v) +i\left( \frac{\partial }{\partial y}(\Delta w)-
  \frac{\partial }{\partial x}(\Delta v) \right) =0.$ A harmonic
function (in particular the real and imaginary part respectively of
$\frac{\partial u}{\partial\bar{z}}$) vanishing on a dense open subset
vanishes identically thus $\frac{\partial u}{\partial\bar{z}}$
vanishes identically on $\Omega$ i.e.\ $u$ is holomorphic on $\Omega$
meaning that,
\begin{equation}
0=\Dbar u(z)=\Dbar^2 f(z),\forall z\in \Omega,
\end{equation}
i.e.\ $f$ is bianalytic on $\Omega.$ This proves
Theorem~\ref{radoprop} for $q=2.$
\\
\\
Now we can easily adapt the proof to the cases $q>2.$ Assume $f\in
C^{q-1}(\Omega)$ is $q$-analytic on $U$ and as before let
$u(z):=\Dbar^{q-1} f(z)$.  If $f^{-1}(0)\cap \Omega$ has nonempty
interior then also $\Dbar^q f(z)=0$ on $\left( \Omega\cap
  f^{-1}(0)\right)^{\circ},$ thus $\Dbar u(z) =0$ on
$\Omega\setminus\left(\partial f^{-1}(0)\right),$ which is a dense
open subset of $\Omega.$ Applying the same arguments to $u$ as for the
case $q=2$ we obtain that $\Dbar u(z)$ vanishes identically on
$\Omega,$ in particular $u$ is differentiable\footnote{In fact, $u$
  must be $C^{\infty}$-smooth due to a well-known property of the
  $\Dbar$-operator, see e.g.\ Krantz~\cite[p.\,200]{kr}.}  on $\Omega$
thus $\Dbar^{q-1} f(z)$ is well-defined and differentiable on
$\Omega$, meaning that we can write $\Dbar^q f=
\frac{\partial}{\partial\overline{z}}\left(\Dbar^{q-1}
  f(z)\right)=0,\forall z\in \Omega .$ This completes the proof.
\end{proof}

Note the importance of the starting function $f$ to be $C^{q-1}(\Omega)$ instead of merely continuous, namely, we need in the proof for $q>2$
that $u$ be continuous on $\Omega.$ 
\begin{example}
Let $\Omega:=\{\abs{z}<2\}\subset\C,$ and set,
\begin{equation}
f(z):=
\left\{
\begin{array}{ll}
1-z\bar{z} &, \abs{z}\leq 1, \\
z\bar{z}-1 &, 1< \abs{z}< 2.
\end{array}
\right.
\end{equation}
The function $f$ is clearly $2$-analytic on the open subset
$\Omega\setminus f^{-1}(0)$ (which consists of two disjoint domains),
and we have $f\in C^0(\Omega)$, $0=q-2$. However $f$ is \emph{not}
$2$-analytic at any point of $\{\abs{z}=1\}$. 

Note that this example break \emph{both} the regularity assumption
\emph{and} the boundary maximum modulus principle required in
Theorem~\ref{radoprop}. We are not aware of an example of a function
that only fails one of these two assumptions and that cannot be
extended polyanalytically across its zero set.

\end{example}

As was pointed out by Cartan~\cite{cartan}, an extension of Rad\'o's theorem 
to the case of several variables is easy in the presence of
Hartogs' theorem on separately holomorphic functions.  It turns out
that such a multi-variable Hartogs theorem is indeed known for
polyanalytic functions in the sense of Definition~\ref{avandef}.

\begin{theorem}[Avanissian \& Traoré~{\cite[Theorem 1.3, p.\,264]{avan2}}]\label{hartog1}
  Let $\Omega\subset\Cn$ be a domain and let $z=(z_1,\ldots,z_n),$
  denote holomorphic coordinates in $\Cn$ with $\real z=:x, \imag z=y$. Let
  $f$ be a function which, for each $j$, is polyanalytic of order
  $\alpha_j$ in the variable $z_j=x_j+iy_j$ (in such case we shall
  simply say that $f$ is separately polyanalytic of order
  $\alpha$). Then $f$ is jointly smooth with respect to
  $(x,y)$ on $\Omega$ and furthermore is polyanalytic of order
  $\alpha=(\alpha_1,\ldots,\alpha_n)$ in the sense of Definition~\ref{avandef}.
\end{theorem}

We immediately obtain the following consequence of Theorem~\ref{radoprop}.

\begin{theorem}\label{radoprop1}
  Let $\Omega\subset\Cn$ be a bounded domain with continuous boundary that in each variable separately allows a solution to the Dirichlet problem. Let $(z_1,\ldots
  ,z_n)\in \Cn$ denote holomorphic variables. Let
  $\alpha=(\alpha_1,\ldots,\alpha_n)\in \Z_+^n.$ Then any
  function $f$ which is $C^{\alpha_j-1}$-smooth in the $z_j$ variable, for each $j$, up to the boundary
  and which is a member of $\mathscr{M}_\alpha(\Omega\setminus
  f^{-1}(0))$ is automatically a member of $\mathscr{M}_\alpha(\Omega).$
\end{theorem}

\begin{proof}
  Recall that when $\Omega$ is a complex manifold of same dimension as
  the ambient space, $\mathscr{M}_\alpha$ reduces to a subspace of
  $\alpha$-analytic functions which satisfy the boundary maximum
  modulus principle.  Denote for a fixed $c\in \C^{n-1}$,
  $\Omega_{c,k}:=\{ z\in \Omega :z_j=c_j,j<k, z_j=c_{j-1},j>k
  \}$. Consider the function $f_c(z_k):=$
  $f(c_1,\ldots,c_{k-1},z_k,c_{k},\ldots,c_{n-1})$. Note that the
  restriction of a function in $n$ complex variables which satisfies
  the boundary maximum modulus principle, to a complex one-dimensional
  submanifold, must also satisfy the boundary maximum modulus
  principle.  Clearly, $f_c$ is $\alpha_k$-analytic on
  $\Omega_{c,k}\setminus f^{-1}(0)$ for any $c\in \C^{n-1}.$ Since
  $f^{-1}_c(0)\subseteq f^{-1}(0)$, Theorem~\ref{radoprop} applies
  to $f_c$ meaning that $f$ is \emph{separately} polyanalytic of order
  $\alpha_j$ in the variable $z_j, 1\leq j\leq n$. By
  Theorem~\ref{hartog1} the function $f$ must be polyanalytic of order
  $\alpha$ (in the sense of Definition~\ref{avandef}) on $\Omega.$ By
  Remark \ref{remarkrado}, $f$ satisfies the boundary maximum modulus
  principle. This completes the proof.
\end{proof}

\section{A special subfamily, $\mathfrak{M}_{\alpha}(U),$ of local limits on submanifolds $U\subset\Cn$}

\begin{corollary}[to Proposition~\ref{unifconv}]\label{unifconv1}
  Let $\Omega\subset\Cn$ be a domain and let $\{f_j\}_{j\in \N}$ be a
  sequence of polyanalytic functions of order $\alpha\in \Z_+^n,$ such
  that the $f_j$ converge uniformly on $\Omega.$ Then the limit is a
  polyanalytic function of order $\alpha$ on $\Omega.$
\end{corollary}

\begin{proof}
  Let $(z_1,\ldots,z_n)\in \Cn$ denote holomorphic variables with
  respect to which being polyanalytic is defined. Let $1\leq k\leq n.$
  Fixing all variables except $z_k,$ say $(z_1,\ldots,z_n)=$
  $(c_1,\ldots,c_{k-1},z_k,c_k,\ldots,c_{n-1}),$ for some $c\in
  \C^{n-1}$, we know that for any $j\in \N,$ the restriction of $f_j$
  becomes an $\alpha_k$-analytic function in the variable $z_k$ on the
  set $\Omega_{c,k}:=\{ z\in \Cn:z_i=c_i ,i< k, z_i=c_{i-1},i>k\}.$ By
  Proposition~\ref{unifconv} this implies that the uniform limit
  function, which we denote $f,$ of the sequence $\{f_j\}_{j\in \N},$
  is separately polyanalytic of order $\alpha.$ Thus by Theorem~\ref{hartog1} 
  $f$ is polyanalytic of order $\alpha$ (in the sense of
  Definition~\ref{avandef}) on on $\Omega.$
  This completes the proof.
\end{proof}

We shall in this section be interested in the following families of
functions, which in particular includes $C^q$-smooth boundary values of
special $\alpha$-analytic functions ($\abs{\alpha}\leq q$).

\begin{definition}
  Let $V\subset\Cn$ be a domain. Denote by $\mathfrak{M}(V)$ the set
  of countably analytic functions $g$
  which obey the one dimensional boundary maximum modulus property of
  Definition~\ref{bmp} and such that
  $1/g$ has the \textsc{1d-bmmp} on $V\setminus g^{-1}(0)$. 
  Let $U\subset\Cn$ be a real submanifold. Denote by
  $\mathfrak{M}_{\alpha}(U)$ the set of functions $f$ defined on $U$
  with the property that for every $p\in U$ there exists an open set
  $V_p\subset\Cn$ such that $f$ can be uniformly approximated on $U\cap
  V_p$ by $\alpha$-analytic functions in
  $\mathfrak{M}(V_p)$.\footnote{Note that the reciprocals of the 
  approximating functions need not be $\alpha$-analytic, merely
  countably analytic away from their singularities.}
\end{definition}

\begin{example}\label{thx}
  Clearly if $M\subset\Cn$ is an open subset then so is any open
  $U\subset M$ thus, by Proposition~\ref{unifconv},
  $\mathfrak{M}_{\alpha}(U)$ coincides with the set of
  $\alpha$-analytic functions on $U$ which can be locally uniformly
  approximated by $\alpha$-analytic functions that together with their
  reciprocals (where defined) satisfy the one dimensional boundary
  maximum modulus property. If $\Omega\subset\Cn$ is a \emph{complex}
  submanifold (of dimension $\leq n$) then
  $\mathscr{O}(\Omega)\subseteq \mathfrak{M}_{\alpha}(\Omega).$
\end{example}

\begin{example}
  For any open subset $V\subset \Cn,$
  \begin{equation}
    \mathfrak{M}_{(1,\ldots ,1)}(V)=\mathscr{O}(V)
  \end{equation}
  and, with the usual partial ordering of multi-indices (i.e.\
  $\alpha\leq \beta$ if $\alpha_j\leq\beta_j,1\leq j\leq n$), we have
  for any real submanifold $U\subset \Cn,$
  \begin{equation}
    \alpha,\beta\in \Z_+^n ,\alpha\leq \beta\Rightarrow
    \mathfrak{M}_{\alpha}(U)\subseteq \mathfrak{M}_{\beta}(U). 
  \end{equation}
  In particular, the set of restrictions of holomorphic functions to
  $U$ belong to $\mathfrak{M}_{\alpha}(U),$ for any $\alpha\in
  \Z_+^n.$
\end{example}

\begin{example}
  If $U\subset\Cn$ is a generic $CR$ submanifold (see the Appendix)
  then, due to a theorem of Baouendi \& Treves~\cite{bt1} (the special
  case we need is formulated more directly in Boggess \&
  Polking~\cite[Theorem 2.1, \,p.761]{bp}),
  \begin{equation}
    \mathfrak{M}_{(1,\ldots,1)}(U)=\left\{ \text{continuous $CR$ functions on } U\right\},
  \end{equation}
  (for the definition of continuous $CR$ functions on $U,$ see Appendix). 
\end{example}

In contrast to the case of $\alpha$-analytic functions where
$\alpha_j\leq 1,1\leq j\leq n,$ we know that there exists nonconstant
$\alpha$-analytic functions which are real valued as soon as at least
one $\alpha_j >1.$

\begin{example}
  Let $U\subset \Cn$ be a domain, let $(z_1,\ldots,z_n)$ denote
  holomorphic coordinates in $\Cn$ with respect to which being
  polyanalytic is defined and let $\alpha\in \Z_+^n$ such that at
  least one $\alpha_j >1$. Then $z_j\overline{z}_j$ is a real-valued
  function belonging to $\mathfrak{M}_{\alpha}(U).$ More generally let
  $P(z)=\sum_{\abs{\beta}<q} a_{\beta}z^{\beta}$. Then the function
  $\abs{P(z)}^2=\overline{P(z)}\cdot P(z)$ is a polyanalytic function
  of order $\alpha$ for some $\alpha$ with $\abs{\alpha}\leq q.$ Since
  holomorphic polynomials obey the maximum principle (and so do their
  reciprocals where well-defined) on any complex submanifold of $U$ we
  obtain $\abs{P(z)}^2\in \mathfrak{M}_{\alpha}(U).$\footnote{We
    mention that it is straightforward to further define
    $\mathfrak{M}_{(\infty,\ldots,\infty)}(U)$ for countably analytic
    functions and in that case, $\{ g:g=\abs{f(z)}^k,\mbox{ for some
    }f\in\mathscr{O}(U)\}\subset
    \mathfrak{M}{(\infty,\ldots,\infty)}(U).$ We shall however only
    need $\mathfrak{M}_{\alpha}(U)$ for finite order $\alpha$ in this
    text.  }
\end{example}

\begin{example}
  It is also clear that $\mathfrak{M}_{\alpha}(U)$ always contains
  functions which are neither holomorphic nor plurisubharmonic, e.g.\
  the restriction of $(z^5e^{z^2})\cdot\overline{z}^3$ to any
  submanifold $U\subset \C$ belongs to $\mathfrak{M}_4(U).$
\end{example}

\begin{example}
  For $\alpha\in \Z_+^n$ such that at least one $\alpha_j >1$,
  $\mathfrak{M}_{\alpha}$ is in general not closed under
  addition. Take e.g.\ $\alpha=(2,3),n=2$. Then
  $g(z_1,z_2)=(z_1z_2^2){\bar{z_1}} \bar{z}_2^2\in
  \mathfrak{M}_{(2,3)}(\C^2),$ but $(1-g)\notin
  \mathfrak{M}_{(2,3)}(\C^2)$. 
\end{example}

\begin{example}
  The phenomenon that for $U\subset \Cn,$
  $\mathfrak{M}_{(1,\ldots,1)}(U)$ may contain elements that are not
  the restrictions of holomorphic functions, also holds for
  $\alpha\geq (1,\ldots,1)\in \Z_+^n.$ Let $(z_1,\ldots,z_n)\in \Cn,
  n>1,$ be complex coordinates with respect to which being
  $\alpha$-analytic is defined.  $U$ be the real analytic hypersurface
  $U:=\{ z\in \Cn: \imag z_n=0\}.$ It is known (see e.g.\
  Boggess~\cite[\,p.109]{b4}) that any continuous function $f$ on $U$
  which is holomorphic with respect to $z_1,\ldots,z_{n-1},$ can be
  locally uniformly approximated on $U$ by entire functions.  For
  example given any point $p\in U$ there exists an open neighborhood
  $p\in V_p\subset \Cn$ together with a sequence $\{E_{j,p}\}_{j\in
    \N}$ of holomorphic functions such that $E_{j,p}\to f$ on $V_p\cap
  U.$ Take e.g.\ the continuous function $f(z):= \abs{\real z_n}\cdot
  \exp(\sum_{j=1}^{n-1}z_j),z\in U$ and set $g(z)=f(z)\cdot
  \overline{P(z)},z\in U$ where $P(z):=\tilde{P}|_U$ for a holomorphic
  polynomial $\tilde{P}(z)$ of highest power $(\alpha_j-1)$ in the
  variable $z_j,1\leq j\leq n.$ Then on any $V_p\cap U$ as above the
  function $g$ is the uniform limit of $\{ E_{j,p}\cdot
  \overline{P}\}_{j\in \N}$ where clearly each $E_{j,p}\cdot
  \overline{P}$ is $\alpha$-analytic on $V_p,$ thus $g\in
  \mathfrak{M}_{\alpha}(U)$ since $\abs{E_{j,p}\cdot \overline{P}}=$
  $\abs{E_{j,p}}\cdot \abs{P},$ and both $E_j$ and $P$ are holomorphic
  on $V_p.$ However the factor $\abs{\real z_n}$ implies that $g(z)$
  is not the restriction to $U$ of an $\alpha$-analytic function on
  any neighborhood of the origin.
\end{example}

\section{Implication of constant modulus}

We shall now point out that some results for the space
$\mathfrak{M}_{\alpha}(M)$ where $M\subset \Cn$ is a generic $CR$
submanifold (see Appendix), follow immediately from the construction
of one dimensional manifolds attached near a reference point and which
cover an open subset. For clarity we shall begin with the easier case
of hypersurfaces and then give a generalization to higher
codimension. In the case of hypersurfaces, the main tool follows from a
result of Lewy~\cite{lewy} and in higher codimension, the
generalization (involving the so called Levi cone) is given by Boggess
\& Polking~\cite{bp} (see Boggess~\cite{b4} for textbook version).

\subsection{The case of $C^3$-smooth hypersurfaces near $1$-convex points}

\begin{proposition}\label{propettconv}
  {\it Let $M\subset \Cn,n\geq 2,$ be a $C^3$-smooth hypersurface such that
  the Levi form of $M$ at the origin has at least one positive
  eigenvalue.  Let $f\in \mathfrak{M}_{\alpha}(\Omega)$ on a domain
  $\Omega\subset M$ containing the origin.  Then,
\begin{itemize}
\item [(i)] There exists an open subset of $\Cn$ containing an open
  $M$-neighborhood of the origin in its closure, to which $f$ extends
  to a $\alpha$-analytic function (i.e.\ $f$ can be identified near
  the origin as the trace values of an $\alpha$-analytic function).
\item [(ii)] $\abs{f}$ cannot be constant on a domain $0\in
  \omega\subset \Omega$ unless $f$ is near $0$ the trace of some
  $\alpha$-analytic function of the form
  $\lambda\cdot\frac{\overline{Q(z)}}{Q(z)}$ for some constant
  $\lambda$ and holomorphic polynomial $Q(z)$ of degree $<\alpha_j$ in
  $z_j,1\leq j\leq n$.
\end{itemize}
}
\end{proposition}

\begin{proof}
  Let $z=(z_1,\ldots,z_n)\in \Cn$ denote the complex variables with
  respect to which being $\alpha$-analytic is defined.  Given a
  unitary matrix $U$, the linear coordinate transformation $z\mapsto
  Uz$ map polynomials in the components of $\bar{z}$ with holomorphic
  coefficients to polynomials in the components of $\overline{Uz}$
  with holomorphic coefficients. In particular for any fixed
  multi-index $\alpha\in \Z_+^n$ and any unitary matrix $U$, there is
  a fixed $\beta\in \Z_+^n$ such that every function $P$ that is
  $\alpha$-analytic with respect to $z$ becomes $\beta$-analytic with
  respect to $Uz$ and 
  $\overline{U}^T$ transforms the coordinates back to $z$ making
  $P(z)=P(\overline{U}^T Uz),$ $\alpha$-analytic.  The proof is based
  on first finding analytic discs attached to $\omega$ and then
  filling out a one-sided open subset of $\Cn$ containing $\omega$ in
  its closure. To prove the existence of the discs we shall first make
  a linear coordinate transformation using a unitary matrix.  Once we
  have the discs, an inverse transformation to $z$ will preserve the
  existence of discs attached near a reference point, because a
  complex one dimensional manifold attached to $\omega$ remains a one
  dimensional manifold attached to $\omega$ after a linear change of
  coordinates. Then we shall use the properties of $f$ (in the
  $z$-variables) in order to obtain that a sequence uniformly
  approximating $f$ near $0$ also converges on the one-sided open
  subset.
  We will need the following lemma.
\begin{lemma}\label{oobb}
  There is a domain $V\subset \Omega$ (chosen sufficiently small)
  containing $0,$ decomposed by $M$ into $V^+,V\cap M,V^-$ with
  $\omega= V\cap M$ containing $0$ such that:
  \begin{enumerate}
  \item There exists a function $F$ which is $\alpha$-analytic on
    $V^+$ such that $F|_{\omega}=f.$
  \item $\max_{\omega} \abs{f}\geq \max_{\overline{V^+}} \abs{F}.$ (In
    particular if $\abs{f}\equiv C$ on $\omega$ then
    $\max_{\overline{V^+}} \abs{F}\leq C$ on $\overline{V^+}$).
  \end{enumerate}
\end{lemma}

\begin{proof}
  The method of proof is that of Lewy~\cite{lewy} and uses filling by
  interiors of analytic discs attached to $M$ near $p$. We describe
  the method in the case of hypersurfaces, based on analytic discs.
  For further details, see e.g.\ Boggess~\cite[\,p.209]{b4}.

  We start from the local graph representation of $M\cap V=\{\imag
  z_n=h(z_1,\ldots ,z_{n-1},\real z_n)\}\cap V$ for a sufficiently
  small open neighborhood $V$ of the origin in $\Cn.$ After a
  change of coordinates, $z\mapsto Uz=(w,\tilde{z})^T\in \C^{n-1}\times \C,
  \tilde{z}=x+iy$, we can diagonalize the Hermitian symmetric matrix
  $S:=\left[\frac{\partial^2h(0)}{\partial w_j\partial \bar{w}_k}
  \right]_{(n-1)\times (n-1)}$ such that $\Lambda:= \bar{U}^T S U
  =$diag$(\lambda_{1},\ldots, \lambda_{n-1})$ for some unitary matrix
  $U$. We may assume that
  \[
  h(w_1,\overbrace{0,\ldots,0}^{n-2},0)=\lambda_{1}\abs{w_1}^2+o(\abs{w}^2,\abs{x}^2) \]
  (where $o(\abs{w}^2,\abs{x}^2)$ denote terms depending on both $w,x$
  which vanish to third order at $0$).  Since the Levi form at $0$ has
  at least one positive eigenvalue we can assume that $\lambda_1>0$
  (after a reordering of the $w$ coordinates). The manifold $M$
  divides $V$ into $V\cap M,$ $V\cap \{y>h(w,x)\}$ and $V\cap
  \{y<h(w,x)\}.$ For a domain $0\in \omega\subset M$, a sufficiently
  small translation of the complex line $\C \times 0 \subset \C
  \times \C^{n-1}$ in the positive $y$-direction intersects
  $\{y>h(w,x)\}$ in a simply connected open subset. More precisely, there
  exists $\epsilon,\delta >0$ such that
  \begin{equation*}
    \begin{split}
      A_{x,y,w_2,\ldots,w_{n-1}} &:= \{(\zeta,w_2,\ldots,w_{n-1},\tilde{z}):\zeta\in
      \C\} \cap \\
      & \qquad
      \{\abs{y}<\epsilon,\abs{x},\abs{w_2},\ldots,\abs{w_{n-1}}<\delta\}
    \end{split}
  \end{equation*}
  is simply connected with its boundary contained in
  $\omega$. Finally, the union of the $A_{x,y,w_2,\ldots,w_{n-1}}\cap
  \{y>h(w,x)\}$ cover an ambient open subset (see example in Figure 1)
  which we denote by $\widetilde{U}.$ Now we transform the coordinates
  back to $(z_1,\ldots,z_n)$ by applying the linear transformation
  given by the matrix $\overline{U}^T.$ The image of each
  $A_{x,y,w_2,\ldots,w_{n-1}}\cap \{y>h(w,x)\}$ is again a complex
  one-dimensional submanifold attached to $\omega$. Furthermore the
  image of $\widetilde{U}$ is an open subset of $\Cn$ containing
  $\omega$ in its closure.  Without loss of generality, we may assume
  that this open subset is $V^+$.

\begin{figure} 
\centering
\includegraphics[scale=0.6]{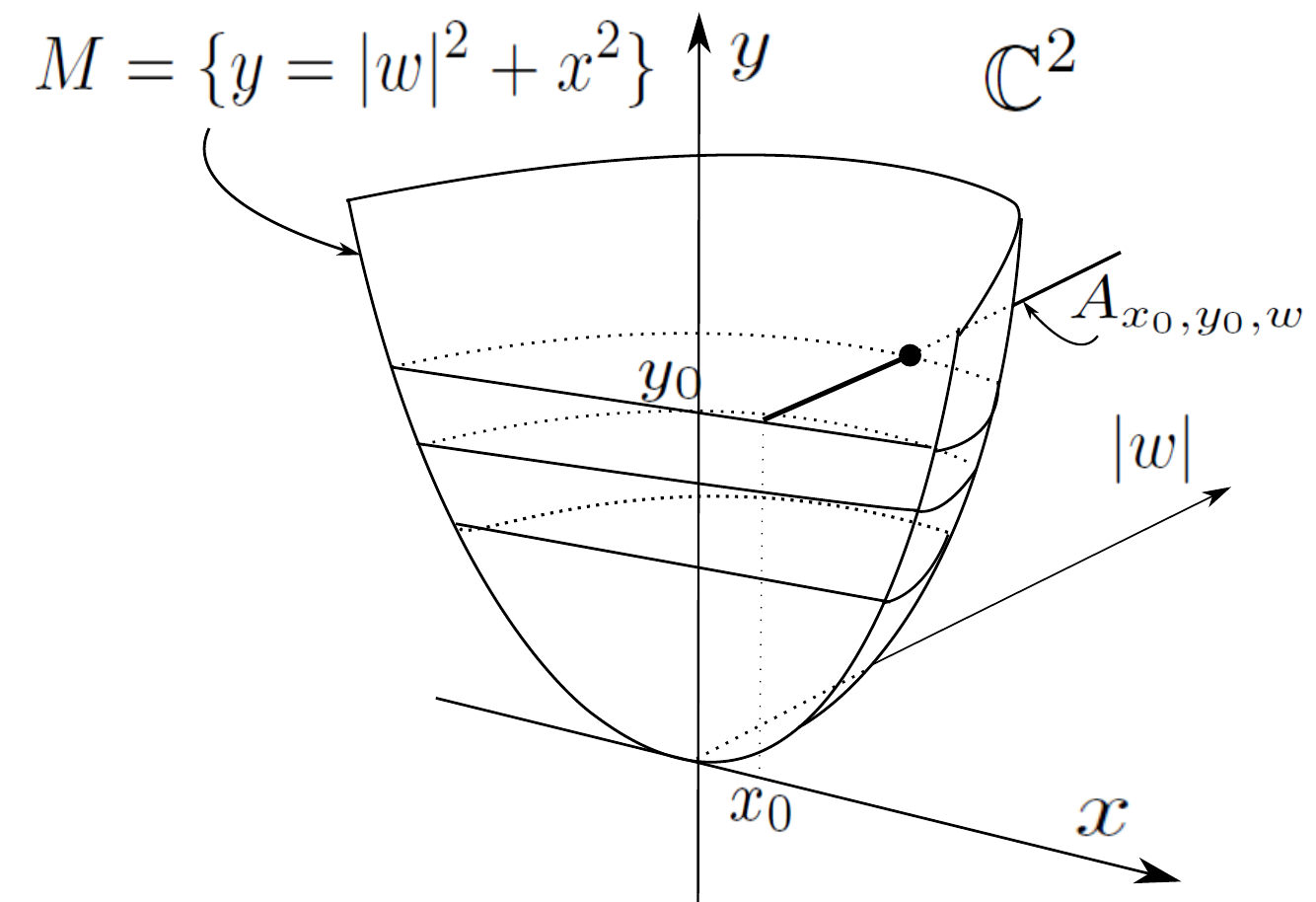}
\caption{Example of filling by one-dimensional complex manifolds (analytic discs attached to
a strictly pseudoconvex hypersurface in $\C^2$. Each $A_{x_0,y_0,w}$ denotes a real two-dimensional manifold, which is also 
a complex manifold of complex dimension one.}
\end{figure}

Now by definition, any function in $\mathfrak{M}_{\alpha}$ can be
approximated locally uniformly on $\omega$ by functions, say
$\{P_j\}_{j\in \N}$, that are $\alpha$-analytic on an open (in $\Cn$)
neighborhood of $\omega$. Furthermore, the restriction of these
functions to each analytic disc $\psi(\overline{D})$ is a function
$P_j\circ \psi$ that obeys the maximum modulus principle. (Here, by
slight abuse of notation, $\psi(\overline{D})$ is the
image, under the reverse transformation with matrix the $\overline{U}^T$, of
the closure of sets $A_{x,y,w_2,\ldots,w_{n-1}}\cap \{y>h(w,x)\}$.)

Thus the $P_j$ converge uniformly on the union of the interiors of the
analytic discs, which contains the open one-sided neighborhood $V^+$
(or in the $Uz$-variables, the part part of $V$ belonging to
$\{y>h(w,x)\}$) and by Theorem~\ref{unifconv} the limit function,
which we denote by $F$, is $\alpha$-analytic on $V^+$ %
Also, the inequalities $\max_{\overline{V^+}}\abs{P_j}\leq$
$\max_{\overline{\omega}}\abs{P_j},$ $\forall j\in \N,$ immediately
imply that $\max_{\overline{V^+}}\abs{F}\leq$
$\max_{\overline{\omega}}\abs{f}.$ This completes the proof of Lemma
\ref{oobb}.
\end{proof}
Part (1) of Lemma~\ref{oobb} proves statement (i).  Also
Lemma~\ref{oobb} immediately takes care of the case $C=0$. We may
therefore assume that
$C>0$.
We continue to work with $(z_1,\ldots,z_n)\in \Cn$ as the coordinates
in $\Cn$ with respect to which being $\alpha$-analytic is defined.

\begin{lemma}\label{hack}
  The extension $F$ in (1) of of Lemma~\ref{oobb} must have constant
  modulus as soon as its trace $F|_{\omega}$ has constant modulus.
\end{lemma}

\begin{proof}
  Assume that there exists an open subset $W\subset V^+$ such
  $\abs{F}\equiv C$ on $W.$ By Theorem~\ref{ballcar} the restriction
  of $F$ to any \[ W_{k,c}:=W\cap \{z_j=c_j,j<k,z_j=c_{j-1},j>k\} \]
  for a fixed $c\in \C^{n-1}$, is a function of the form
  $\lambda_{k,c}\cdot \overline{Q_{k,c}(z_k)}/Q_{k,c}(z_k)$, where
  $Q_{k,c}$ is a polynomial of degree less than $\alpha_k$ and
  $\lambda_{k,c}$ is a complex constant. By continuity of $\abs{F}$ up
  to $\omega$ we deduce that $\abs{\lambda_{k,c}}$ must equal $C$
  independently of $k,c$. Now any open subset has a condensation point
  of order $\alpha_k$ so by Theorem~\ref{condens} the restriction of
  $F$ extends as $\lambda_{k,c}\cdot
  \overline{Q_{k,c}(z_k)}/Q_{k,c}(z_k)$, to the intersection of any
  $\{z_j=c_j,j<k,z_j=c_{j-1},j>k\},$ $1\leq k\leq n$ with $V^+$ which
  passes $W.$ Repeated application of Theorem~\ref{condens} starting
  from that intersection we obtain (using a union of intersections of
  $V^+$ with polydiscs) that the restriction of $F$ to $V^+\cap
  \{z_j=c_j,j<k,z_j=c_{j-1},j>k\},$ $1\leq k\leq n,$ has the form
  $\lambda_{k,c}\cdot \overline{Q_{k,c}(z_k)}/Q_{k,c}(z_k)$. Hence
  $\abs{F}\equiv C$ on $V^+$ and $F$ is an $\alpha$-analytic extension
  of $f$ from $\omega$ so we are done under the assumption that
  $\abs{F}\equiv C$ on some open subset of $V^+.$

  Now, we assume instead that $\abs{F}< C$ on a dense (necessarily
  open) subset of $V^+.$ This implies that there exists an open subset
  $W\subset V^+$ on which $\frac{C}{2}<\abs{F}<C$ such that $\partial
  W\cap \omega$ is relatively open. In particular %
  $\frac{1}{F}$ is well-defined on $\overline{W}$ and since
  $\frac{C}{2}>0$ we can choose a subsequence $\{ P_{j_{l}}\}_{l\in
    \N}$ such that each $P_{j_l}$ is nowhere zero on $W$ and
  $1/P_{j_{l}}\to 1/F$ on $W.$ Hence,
  \begin{equation}\label{hooho}
    \max_{z\in \overline{W}} \abs{\frac{1}{F(z)}}\leq \frac{1}{C}.
  \end{equation}
  But we already know $\max_{z\in \overline{W}}\abs{F(z)}\leq C$
  (where we are using $W\subset V^+$ together with (2) of
  Lemma~\ref{oobb}). So if there were to exist a point $z_0\in
  \overline{W}$ such that $\abs{ F(z_0)} \neq C$, then necessarily
  $\abs{F(z_0)}<C$ which implies $\abs{\frac{1}{F(z_0)}}> \frac{1}{C}$
  (thus incompatible with Equation~\eqref{hooho}). This yields
  $\abs{F}\equiv C$ on $W$, which is a contradiction. This completes
  the proof of Lemma~\ref{hack}.
\end{proof}

Finally we need to verify that the $\alpha$-analytic extension $F$
which has constant modulus in fact must have the form required in
(ii).

\begin{lemma}\label{hack2}
  Let $\mathcal{V}\subset\Cn$ be an open subset with variables
  $z=(z_1,\ldots,z_n)\in \Cn$.  Let $F(z)$ be a function
  $\alpha$-analytic with respect to $z$ such that $\abs{F}\equiv C$ on
  $\mathcal{V}$ for some constant $C>0.$ Then
  $F(z)=\lambda\cdot\frac{\overline{Q(z)}}{Q(z)}$ for some constant
  $\lambda$ and holomorphic polynomial $Q(z).$
\end{lemma}

\begin{proof}
  Let $\mathcal{V}_{c,k}=\mathcal{V}\cap \{
  z_j=c_j,j<k,z_j=c_{j-1},j>k\}.$ By Theorem~\ref{ballcar} the
  restriction to each $\mathcal{V}_{c,k}$ satisfies
  $F(z)=\lambda_{c,k} \overline{Q_{c,k}(z_k)}/Q_{c,k}(z_k),$
  $\abs{\lambda_{c,k}}=C,$ where each $Q_{c,k}(z_k)$ is a holomorphic
  polynomial in one variable of degree $<\alpha_k.$ Obviously the
  function $Q(z)$ defined by,
  \begin{equation}\label{nn0}
    Q(z)=Q_{c,k}(z_k), 
    \forall z\in \mathcal{V}_{c,k}, \forall c\in \C^{n-1},1\leq k\leq n,
  \end{equation}
  is locally bounded and separately a holomorphic polynomial in each
  variable, in particular of order $\alpha_k-1$ in the $z_k$-variable,
  $1\leq k\leq n$.  This implies that $Q(z)$ is jointly a
  polynomial, i.e.\ that it has the representation
  \begin{equation}
    Q(z)=\sum_{\{\beta\in \N^n : \beta_j<\alpha_j ,1\leq j\leq n\}} a_{\beta} z^{\beta}, z\in \mathcal{V}, \text{ some constants } a_{\beta}.
  \end{equation}
  Similarly we conclude that the function, 
  \begin{equation} 
    \lambda(z)=\lambda_{c,k}, 
    \forall z\in \mathcal{V}_{c,k}, \forall c\in \C^{n-1},1\leq k\leq n,
  \end{equation}
  must be holomorphic and since $\abs{\lambda}\equiv C$ on an open
  subset, $\lambda\equiv\text{constant}$ on $\mathcal{V}.$ This means that
  the function $F(z)$ coincides pointwise with the function $\lambda
  \cdot \frac{\overline{Q(z)}}{Q(z)}$ on the open $\mathcal{V}.$ This
  proves Lemma~\ref{hack2}.
\end{proof}

This also completes the proof of Proposition~\ref{propettconv}.
\end{proof}

\subsection{$C^4$-smooth $CR$ submanifolds of $\Cn$ near points with a Levi cone condition}\label{rigid}

The proof of Proposition~\ref{propettconv} relies heavily upon Lemma
\ref{oobb}.  The adaptation to higher codimension follows from a
difficult technical improvement of Lewy's theorem (see Boggess \&
Polking~\cite{bp}), but for our purposes we only require the existence
of a family of analytic discs attached near a reference point. This
result is known (we have chosen to place this citation in the
Appendix) in the case of a $C^4$-smooth $CR$ submanifold $M\subset
\Cn$ near a reference point where the Levi cone has nonempty
interior\footnote{This generalizes the case of a $1$-convex point (by
  which we mean that the Levi form has at least one positive
  eigenvalue) for hypersurfaces, see Appendix.}, see
Lemma~\ref{bpthm}.  A consequence of Lemma~\ref{bpthm} is the
following.
\begin{remark}
  If the convex hull of the image of the Levi form at $p\in M$
  contains an interior point, there exists an open subset
  $\mathcal{V}\subset \Cn$ and a domain $p\in \omega \subset M$ such
  that (i) $\omega\subset \overline{\mathcal{V}}$, (ii) $\mathcal{V}$
  can be covered by the interiors of analytic discs whose boundaries
  are attached to $\omega.$ This immediately yields for $f\in
  \mathfrak{M}_{\alpha}(U)$ with a domain $U\subset M,$ that there
  exists a subdomain $\omega\subset U$ with $p \in \omega$ and an
  associated $\mathcal{V}$ such that
  max$_{z\in\overline{\omega}}\abs{f(z)} \geq$max$_{z\in
    \overline{\mathcal{V}}}\abs{F(z)},$ where $F$ is $\alpha$-analytic
  on $\mathcal{V}$ and $F|_{\omega} =f.$
\end{remark}

\begin{proposition}\label{jo}
  {\it Let $M\subset \Cn,n\geq 2,$ be a $C^4$-smooth $CR$ submanifold and
  let $p\in M$ such that the image of the convex hull of the Levi cone
  at $p$ has nonempty interior. Let $U\subset M$ be an open
  neighborhood of $p$ and $f\in \mathfrak{M}_{\alpha}(U).$ Then
  $\abs{f}$ cannot be constant on a neighborhood of $p$ unless $f$
  extends to an $\alpha$-analytic function of constant modulus (in
  particular of the form $\lambda\frac{\overline{Q(z)}}{Q(z)}$ for
  some holomorphic polynomial $Q(z)$ of degree $<\alpha_j$ in
  $z_j,1\leq j\leq n$) on an open $\mathcal{V}\subset \Cn$ such that
  $\overline{\mathcal{V}}$ contains an $M$-neighborhood of $p$.}
\end{proposition}

\begin{proof}
  By Lemma~\ref{bpthm} we obtain that Lemma~\ref{oobb} still holds
  if we replace $M$ by a $C^4$-smooth generic submanifold of
  $\Cn$ (of arbitrary positive codimension) and replace $p$ by a point
  in $M$ such that the Levi cone at $p$ has nonempty interior.  The
  remaining arguments of the proof can now be repeated analogously to
  the proof of Proposition~\ref{propettconv}.
\end{proof}

\begin{example}
  Let $M\subset \Cn$ be a generic $C^4$-smooth submanifold and let
  $p\in M$ be such that the interior of the Levi cone at $p$ is
  nonempty.  Proposition~\ref{jo} shows that every continuous $CR$
  function $f$ whose modulus is constant on a neighborhood of a point
  $p\in M$ must have a holomorphic extension (to some open subset, but
  not necessarily a full neighborhood of $p$), again of constant
  modulus. Since holomorphic functions of constant modulus must be
  constant, this implies that $f$ must reduce to a constant near $p$
  on $M.$ This result is known due to Range~\cite{range1} for
  $C^{\infty}$-smooth $M$ and for $p$ of so called finite
  type\footnote{If a generic $CR$ submanifold $M\subset\Cn$ has real
    codimension $d$ then there is associated to every point $p\in M$ a
    set of $d$ integers $m_1,\ldots,m_d$ called Hörmander numbers
    (whose definition involves higher order Lie brackets). $p$ is said
    to be of finite type if all the Hörmander numbers are finite, see
    e.g.\ Boggess~\cite[\,p.181]{b4}.}, but with a different proof. We
  mention that Stoll~\cite{stoll} also deduces a version of this
  result for $C^{\infty}$-smooth boundaries
  but where the method of proof depends on a result of Hakim \&
  Sibony~\cite{haksib}, which in turn depends on working with
  functions $C^{\infty}$ on $\overline{\Omega}$ and in this proof, the
  condition on the boundary being $C^{\infty}$-smooth is necessary,
  see Krantz ~\cite[\,p.5]{kr2010}.
\end{example}



\begin{appendix}
\section{Preliminaries for Lewy's theorem}
A complex structure $J$ on $T \Cn$ 
is defined as a real linear map $J:T\Cn\to T\Cn$
such that $J^2=-Id,$ specifically $J$ is defined fiberwise on the tangent vector spaces by $\R$-linear maps $J_p: T_p \Cn\to T_p \Cn.$
If $M\subset \Cn$ is a submanifold, 
$T^c_p M:=T_p M\cap J_p(T_p M)$ is called
the holomorphic tangent space of $M$ at $p.$ $J$ maps each $T^c_p M$ to itself thus defines a complex structure on $T^c_p M.$
If $T^c_p M$ has constant dimension ($CR$ dimension) for every $p$ then $M$ is called a $CR$ manifold. 
$M$ is called generic if $T_p \Cn =T_p M\oplus J_p(T_p M/T^c_p M).$
The $\R$-linear maps $J_p: T_p^c M\to T_p^c M$
have eigenvalues $\pm i.$
$J$ extends to a $\C$-linear map on the bundle $\C\otimes T^c M=\bigcup_{p\in M} \C\otimes T_p^c M $ such that the extension again has eigenvalues $\pm i.$ This decomposes $\C\otimes T^c M=H^{1,0} M\oplus H^{0,1} M$ namely a $\C$-linear and anti-$\C$-linear part, where we denote by $H^{0,1} M$ the 
anti-$\C$-linear part. 
\begin{definition}
A differentiable function $f$ on $M$ is called $CR$ if it is annihilated by any $C^1$-section $X$ of $H^{0,1} M$ over $M.$
A distribution $f$ is called $CR$ if $Xf=0$ in the weak sense
i.e.\ $\langle f,X^{\mbox{adj}}\phi\rangle =0,\forall \phi\in C^{\infty}_c(M),$
where $X^{\mbox{adj}}$ denotes the adjoint.
\end{definition}
The Levi form at $p$ of $M,$ denoted $\mathcal{L}_{p}$ is defined as
the map $H^{0,1}_p M\to T_p M/T^c_p M,$ $\mathcal{L}_p(X)
=\frac{1}{2i}[\tilde{X},\overline{\tilde{X}}]|_p$mod$\C\otimes T^c_p
M, X\in H^{0,1}_p M,$ where $\tilde{X}$ is any local ambient extension
with $\tilde{X}|_p=X.$ However for practical reasons one usually identifies the image of the Levi form as a subspace of $N_p M:=J(T_p M/T_p^c M)$.
$N_p M$ is called the normal space at $p$ and
is a real manifold with dimension the same as the real codimension of
$M.$ Denote by $\Gamma_p (\subseteq N_pM)$ the cone which constitutes
the convex hull of the image of $\mathcal{L}_p$ (see Boggess \&
Polking~\cite{bp}).  For two subcones $\Gamma_1,\Gamma_2$ of
$\Gamma_p$ we say that $\Gamma_1$ is smaller that $\Gamma_2$ if
$\Gamma_1\cap S_p$ (where $S_p$ denotes the unit sphere in $N_p M$) is
a compact subset of the (relative) interior of $\Gamma_2 \cap S_p.$
Boggess \& Polking~\cite{bp} proved a theorem on holomorphic extension
of continuous $CR$ functions near a point such that the Levi cone at
$p$ has nonempty interior. The domain of extension has the shape of
the product of an open set with a cone.  The proof involves explicit
construction of families of analytic discs by solving a Bishops
equation, such that the center of these discs pass each point of an
open subset of the given normal cone and simultaneously are attached
sufficiently close to $p.$ For our purposes it is precisely the
existence of such families of discs which is useful.

\begin{lemma}[Boggess~{\cite[\,p.207]{b4}} ]\label{bpthm} 
  Let $M\subset \Cn$ be a $C^l,l\geq 4,$ generic embedded $CR$
  submanifold and let $p\in M$ be a point such that the Levi cone at
  $p$ has nonempty interior.  Then for every open neighborhood
  $\omega$ of $p$ and for each cone $\Gamma <\Gamma_p,$ there is a
  neighborhood $\omega_{\Gamma}\subset\omega$ and a positive number
  $\epsilon_{\Gamma}$ such that each point in $\omega_{\Gamma}
  +\{\Gamma\cap B_{\epsilon_{\Gamma}}\}$ is contained in the image of
  an analytic disc whose boundary image is contained in $\omega.$
\end{lemma}
\end{appendix}

\end{document}